\documentclass[10pt,a4paper]{amsart}
\usepackage{hyperref,amsmath,graphics}
\newtheorem{theorem}{Theorem}[section]
\newtheorem{remark}[theorem]{Remark}

\newtheorem{thmx}{Theorem}[section]
\theoremstyle{theorem}
\newtheorem{example}[theorem]{Example}
\newtheorem{lemma}[theorem]{Lemma}
\theoremstyle{definition}

\numberwithin{equation}{section} \makeatletter
\numberwithin{equation}{section} \makeatletter
\@namedef{subjclassname@2010}{ 2010 Mathematics Subject
Classification} \makeatother
\title[Annulus for the zeros of a polynomial]{On annulus containing all the zeros of a polynomial}
\author{\footnotesize{N. A. Rather and Suhail Gulzar}}
 \address{Department of Mathematics \\
    University of kashmir \\
   Srinagar, Hazratbal 190006
   \\ India}

\begin{document}
\date{}
\maketitle
\begin{center}
    \footnotesize{ Department of Mathematics, University of Kashmir  Hazratbal\\
     Srinagar 190006, India}\\
     emails: \texttt{dr.narather@gmail.com},     \texttt{sgmattoo@gmail.com},
\end{center}
\begin{abstract}
In this paper, we obtain an annulus containing all the zeros of the polynomial involving binomial coefficients and generalized Fibonacci numbers.  Our result generalize some of the recently obtained results in this direction.
\end{abstract}
\section{\textbf{Introduction and Statements}}
Gauss and Cauchy were the earliest contributors in the theory of the location of zeros of a polynomial, since then this subject has been studied by many people (for example, see \cite{marden,mmr}). There is always a need for better and better results in this subject because of its application in many areas, including signal processing, communication theory and control theory.\\
A classical result due to Cauchy (see \cite[p. 122]{marden}) on the distribution of zeros of a polynomial may be stated as follows:

\begin{thmx}\label{t1}
If $P(z)=z^n+a_{n-1}z^{n-1}+a_{n-2}z^{n-2}+\cdots+a_0$ is a polynomial with complex coefficients, then all zeros of $P(z)$ lie in the disk $|z|\leq r$ where $r$ is the unique positive root of the real-coefficient polynomial
$$Q(x)=x^n-|a_{n-1}|x^{n-1}-|a_{n-2}|x^{n-2}-\cdots-|a_1|x-|a_0|.  $$
\end{thmx}

Recently D\'{i}az-Barrero \cite{db} improved this estimate by identifying an annulus containing all the zeros of a polynomial, where the inner and outer radii are expressed in terms of binomial coefficients and Fibonacci numbers. In fact he has proved the following result.

\begin{thmx}\label{t2}
Let $P(z)=\sum_{j=0}^{n}a_jz^j$ be a non-constant complex polynomial. Then all its zeros lie in the annulus $C=\{z\in\mathbb{C}:r_1\leq |z|\leq r_2 \}$ where
$$r_1=\frac{3}{2}\underset{1\leq k\leq n}{\min}\left\{\dfrac{2^nF_k \binom{n}{k}}{F_{4n}}\left|\frac{a_0}{a_k}\right|\right\}^{\frac{1}{k}},\,\,\, r_2=\frac{2}{3}\underset{1\leq k\leq n}{\max}\left\{\dfrac{F_{4n}}{2^nF_k \binom{n}{k}}\left|\frac{a_{n-k}}{a_n}\right|\right\}^{\frac{1}{k}}. $$
Here $F_j$ are Fibonacci's numbers, that is, $F_0=0,$ $F_1=1$ and for $j\geq 2,$ $F_j=F_{j-1}+F_{j-2}.$ 
\end{thmx} 

More recently, M. Bidkham et. al. \cite{bs} considered $t$-Fibonacci numbers, namely $F_{t,n}=tF_{t,n-1}+F_{t,n-2}$ for $n\geq 2$ with initial condition $F_{t,0}=0,\,F_{t,1}=1$ where $t$ is any positive real number and obtained the following generalization of Theorem \ref{t2}.

\begin{thmx}\label{t4}
Let $P(z)=\sum_{j=0}^{n}a_jz^j$ be a non-constant complex polynomial of degree $n$ and 
$$\lambda_k=\dfrac{(t^3+2t)^k(t^2+1)^n F_{t,k}\binom{n}{k}}{(t^2+1)^k F_{t,4n}}  $$
for any real positive number $t.$ Then all the zeros of $P(z)$ lie in the annulus $R=\{z\in\mathbb{C}:s_1\leq |z|\leq s_2 \}$ where 
$$s_1=\underset{1\leq k\leq n}{\min}\left\{\lambda_k\left|\dfrac{a_0}{a_k}\right|\right\}^{\frac{1}{k}},\quad s_2=\underset{1\leq k\leq n}{\max}\left\{\dfrac{1}{\lambda_k}\left|\dfrac{a_{n-k}}{a_n}\right|\right\}^{\frac{1}{k}}.   $$
\end{thmx}

In this paper, we determine in the complex plane an annulus containing all the zeros of a polynomial involving binomial coefficients and generalized Fibonacci numbers (see \cite{oy}) defined recursively by
\begin{align}\label{fn}\nonumber
F_0^{(a,b,c)}&=0,\,\,\,F_1^{(a,b,c)}=1,\\
F_n^{(a,b,c)}&=
\begin{cases}
a\,F_{n-1}^{(a,b,c)}+c\,F_{n-2}^{(a,b,c)}\quad\textnormal{if n is even,}\\
b\,F_{n-1}^{(a,b,c)}+c\,F_{n-2}^{(a,b,c)}\quad\textnormal{if n is odd,}
\end{cases}
(n\geq 2)
\end{align}
where $a,b,c$ are any three positive real numbers. Our result include Theorems \ref{t2}, \ref{t4} as special cases. More precisely, we prove the following result.

\begin{theorem}\label{th1}
Let $P(z)=\sum_{j=0}^{n}a_jz^j$ be a non-constant complex polynomial of degree $n.$ Then all its zeros lie in the annulus $C=\{z\in\mathbb{C}:r_1\leq |z|\leq r_2 \}$ where
$$r_1=\dfrac{uv+2w}{uvw+w^2}\,\underset{1\leq k\leq n}{\min}\left\{\dfrac{(uvw+w^2)^n u^{\xi(k)}(uv)^{\lfloor \frac{k}{2}\rfloor} F_k^{(u,v,w)}\binom{n}{k}}{F_{4n}^{(u,v,w)}}\left|\dfrac{a_0}{a_k}\right|\right\}^{\frac{1}{k}},$$  $$ r_2=\dfrac{abc+c^2}{ab+2c}\,\underset{1\leq k\leq n}{\max}\left\{\dfrac{F_{4n}^{(a,b,c)}}{(abc+c^2)^n a^{\xi(k)}(ab)^{\lfloor \frac{k}{2}\rfloor} F_k^{(a,b,c)}\binom{n}{k}}\left|\dfrac{a_{n-k}}{a_n}\right|\right\}^{\frac{1}{k}},   $$
 $a,b,c,u,v,w$ are any positive real numbers, $\xi(k):=k-2\lfloor \frac{k}{2}\rfloor$ and $F_m^{(a,b,c)}$ is defined as in \eqref{fn}.
\end{theorem}

\begin{remark}
\textnormal{By taking $a,b,c$ and $u,v,w$ suitably, we shall obtain Theorems \ref{t2}, \ref{t4}. For example, if we take $a=b=u=v=t$ and $c=w=1,$ we obtain Theorem \ref{t4}. }
\end{remark} 

\begin{example}
\textnormal{
We consider the polynomial $P(z)=z^3+0.1z^2+0.3z+0.7,$ which is the only example considered by D\'{i}az-Barrero \cite{db} and by using Theorem \ref{t2}, the annulus containing all the zeros of $P(z)$ comes out to be $0.58< |z|< 1.23$. We improved the upper bound of this annulus by taking $a=1/2,$ $b=1$ and $c=3/8$ in Theorem \ref{th1} and obtained the disk, $ |z|< 1.185, $ which contains all the zeros of polynomial $P(z).$  We can similarly improve the lower bound by choosing $u,$ $v,$ $w$ suitably.}
\end{example}
\section{\textbf{Lemma}}
To prove the above theorem, we need the following lemma.
\begin{lemma}\label{l1}
If $F_k^{(a,b,c)}$ is defined as in \eqref{fn}, then 
\begin{align}\label{le1}
\sum\limits_{k=1}^{n}(ab+c)^{n-k}(ab+2c)^k a^{\xi(k)}(ab)^{\lfloor \frac{k}{2}\rfloor}c^{n-k} F_k^{(a,b,c)}\binom{n}{k}=F_{4n}^{(a,b,c)}
\end{align}
where $\xi(k)=k-2\lfloor \frac{k}{2}\rfloor.$
\end{lemma}
\begin{proof}
For $F_k^{(a,b,c)},$ we have \cite{oy}
$$F_k^{(a,b,c)}=\dfrac{a^{1-\xi(k)}}{(ab)^{\lfloor\frac{k}{2}\rfloor}}\left(\dfrac{\alpha^{k}-\beta^{k}}{\alpha-\beta}\right)  $$
where $\alpha=\frac{ab+\sqrt{(ab)^2+4abc}}{2},$  $\beta=\frac{ab-\sqrt{(ab)^2+4abc}}{2}$ and $\xi(k)=k-2\lfloor \frac{k}{2}\rfloor.$\\
Consider,
\begin{align*}
\sum\limits_{k=1}^{n}&\binom{n}{k}(abc)^{n-k}\big[(ab)^2+abc\big]^{n-k}\big[(ab)^3+2(ab)^2c\big]^ka^{\xi(k)}(ab)^{\lfloor\frac{k}{2}\rfloor} F_k^{(a,b,c)}\\=&\sum\limits_{k=1}^{n}\binom{n}{k}(-1)^{n-k}(\alpha\beta)^{n-k}\Bigg(\sum\limits_{j=0}^{2}\alpha^j\beta^{2-j}\Bigg)^{n-k}\Bigg(\sum\limits_{j=0}^{3}\alpha^j\beta^{3-j}\Bigg)^ka\left(\dfrac{\alpha^{k}-\beta^{k}}{\alpha-\beta}\right)\\=&\dfrac{a\alpha^n}{\alpha-\beta}\Bigg\{\sum\limits_{k=1}^{n}\binom{n}{k}(-1)^{n-k}\Bigg(\sum\limits_{j=0}^{2}\alpha^j\beta^{3-j}\Bigg)^{n-k}\Bigg(\sum\limits_{j=0}^{3}\alpha^j\beta^{3-j}\Bigg)^k\Bigg\}\\&-\dfrac{a\beta^n}{\alpha-\beta}\Bigg\{\sum\limits_{k=1}^{n}\binom{n}{k}(-1)^{n-k}\Bigg(\sum\limits_{j=0}^{2}\alpha^{1+j}\beta^{2-j}\Bigg)^{n-k}\Bigg(\sum\limits_{j=0}^{3}\alpha^j\beta^{3-j}\Bigg)^k\Bigg\}\\=&\dfrac{a\alpha^n}{\alpha-\beta}\Bigg\{\sum\limits_{j=0}^{3}\alpha^j\beta^{3-j}-\sum\limits_{j=0}^{2}\alpha^j\beta^{3-j}\Bigg\}^n-\dfrac{a\beta^n}{\alpha-\beta}\Bigg\{\sum\limits_{j=0}^{3}\alpha^j\beta^{3-j}-\sum\limits_{j=0}^{2}\alpha^{1+j}\beta^{2-j}\Bigg\}^n\\=&a\left(\dfrac{\alpha^n(\alpha^3)^n-\beta^n(\beta^3)^n}{\alpha-\beta}\right)=(ab)^{2n}F_{4n}^{(a,b,c)}.
\end{align*}
Equivalently, we have
\begin{align*}
\sum\limits_{k=1}^{n}\binom{n}{k}(ab+c)^{n-k}(ab+2c)^ka^{\xi(k)}(ab)^{\lfloor\frac{k}{2}\rfloor}c^{n-k} F_k^{(a,b,c)}=F_{4n}^{(a,b,c)}.
\end{align*}
\end{proof}
\section{\textbf{Proof of Theorem}}
\begin{proof}[\textnormal{\textbf{Proof of Theorem \ref{th1}}}]
We first show that all the zeros of $P(z)$ lie in 
\begin{align}\label{p1}
|z|\leq r_2=\underset{1\leq k\leq n}{\max}\left\{\dfrac{(ab+c)^kc^kF_{4n}^{(a,b,c)}}{(ab+c)^n(ab+2c)^k a^{\xi(k)}(ab)^{\lfloor \frac{k}{2}\rfloor}c^n F_k^{(a,b,c)}\binom{n}{k}}\left|\frac{a_{n-k}}{a_n}\right|\right\}^{\frac{1}{k}}
\end{align}
 where $a,b,c$ are any three positive real numbers. From \eqref{p1}, it follows that
\begin{align*}
\left|\frac{a_{n-k}}{a_n}\right|\leq r_{2}^{k}\dfrac{(ab+c)^n(ab+2c)^k a^{\xi(k)}(ab)^{\lfloor \frac{k}{2}\rfloor}c^n F_k^{(a,b,c)}\binom{n}{k}}{(ab+c)^kc^kF_{4n}^{(a,b,c)}},\quad k=1,2,3,\cdots,n
\end{align*}
or 
\begin{align}\label{p2}
\sum\limits_{k=1}^{n}\left|\frac{a_{n-k}}{a_n}\right|\dfrac{1}{r_{2}^{k}}\leq \sum\limits_{k=1}^{n} \dfrac{(ab+c)^n(ab+2c)^k a^{\xi(k)}(ab)^{\lfloor \frac{k}{2}\rfloor}c^n F_k^{(a,b,c)}\binom{n}{k}}{(ab+c)^kc^kF_{4n}^{(a,b,c)}}.
\end{align}
Now, for $|z|>r_2,$ we have
\begin{align*}
|P(z)|=&|a_nz^n+a_{n-1}z^{n-1}+\cdots+a_1z+a_0|\\\geq &|a_n| |z|^n\left\{1-\sum\limits_{k=1}^{n}\left|\frac{a_{n-k}}{a_n}\right|\dfrac{1}{|z|^{k}}\right\}\\>&|a_n| |z|^n\left\{1-\sum\limits_{k=1}^{n}\left|\frac{a_{n-k}}{a_n}\right|\dfrac{1}{r_{2}^{k}}\right\}.
\end{align*}
Using \eqref{le1} and \eqref{p2}, we have for $|z|>r_2,$ $|P(z)|>0.$ Consequently all the zeros of $P(z)$ lie in $|z|\leq r_2$ and this proves the second part of theorem.\\
To prove the first part of the theorem, we will use second part. If $a_0=0,$ then $r_1=0$ and there is nothing to prove. Let $a_0\neq 0,$ consider the polynomial
$$Q(z)=z^nP(1/z)=a_0+a_1z^{n-1}+\cdots+a_{n-1}z+a_n.$$
By second part of the theorem for any three positive real numbers $u,v,w$, if $Q(z)=0,$ then 
\begin{align*}
|z|\leq&\underset{1\leq k\leq n}{\max}\left\{\dfrac{(uv+w)^kw^kF_{4n}^{(u,v,w)}}{(uv+w)^n(uv+2w)^k u^{\xi(k)}(uv)^{\lfloor \frac{k}{2}\rfloor}w^n F_k^{(au,v,w)}\binom{n}{k}}\left|\frac{a_{k}}{a_0}\right|\right\}^{1/k}\\=&\dfrac{1}{\underset{1\leq k\leq n}{\min}\left\{\dfrac{(uv+w)^kw^kF_{4n}^{(u,v,w)}}{(uv+w)^n(uv+2w)^k a^{\xi(k)}(ab)^{\lfloor \frac{k}{2}\rfloor}w^n F_k^{(u,v,w)}\binom{n}{k}}\left|\frac{a_0}{a_{k}}\right|\right\}^{1/k}}\\=&\frac{1}{r_1}.
\end{align*}
Now replacing $z$ by $1/z$ and observing that all the zeros of $P(z)$ lie in 
$$|z|\geq r_1=\underset{1\leq k\leq n}{\min}\left\{\dfrac{(uv+w)^kw^kF_{4n}^{(u,v,w)}}{(uv+w)^n(uv+2w)^k u^{\xi(k)}(uv)^{\lfloor \frac{k}{2}\rfloor}w^n F_k^{(u,v,w)}\binom{n}{k}}\left|\frac{a_0}{a_{k}}\right|\right\}^\frac{1}{k}.$$
This completes the proof of theorem \ref{th1}.
\end{proof}
\noindent\textbf{Acknowledgement} \\
 The second author is supported by Council of Scientific and Industrial Research, New Delhi, under grant F.No. 09/251(0047)/2012-EMR-I.

\end{document}